\documentclass[preprint,12pt]{elsarticle}
\usepackage{amssymb}
\usepackage{amsmath}
\usepackage{amsthm}
\usepackage{caption}
\usepackage{subcaption}
\usepackage[ruled,vlined]{algorithm2e}
\newtheorem{thm}{Theorem}[section]

\newtheorem{lemma}{Lemma}[section]

\newtheorem{obs}{Observation}[section]
\usepackage{graphicx}
  \baselineskip 8 mm

\journal{}

\begin{document}
\begin{frontmatter}


\title{Exploring Algorithmic Solutions for the Independent Roman Domination Problem in Graphs}

\author[inst1]{Kaustav Paul\corref{mycorrespondingauthor}}
\cortext[mycorrespondingauthor]{Corresponding author}
\ead{kaustav.20maz0010@iitrpr.ac.in}

\author[inst1]{Ankit Sharma}
\ead{ankit.23maz0013@iitrpr.ac.in}

\author[inst1]{Arti Pandey\footnote{Research of Arti Pandey is supported by CRG project, Grant Number-CRG/2022/008333, Science
and Engineering Research Board (SERB), India}}
\ead{arti@iitrpr.ac.in}

\affiliation[inst1]{%%organization={Department One},%Department and Organization
            addressline={Department of Mathematics, Indian Institute of Technology Ropar}, 
            city={Rupnagar},
            postcode={140001}, 
            state={Punjab},
            country={India}}

%\author[inst2]{}

%\affiliation[inst2]{%Department and Organization
            %addressline={The Institute of Mathematical Sciences}, 
            %city={Chennai},
            %postcode={600113}, 
            %state={Tamil Nadu},
            %country={India}}

\begin{abstract}
Given a graph $G=(V,E)$, a function $f:V\to \{0,1,2\}$ is said to be a \emph{Roman Dominating function} if for every $v\in V$ with $f(v)=0$, there exists a vertex $u\in N(v)$ such that $f(u)=2$. A Roman Dominating function $f$ is said to be an \emph{Independent Roman Dominating function} (or IRDF), if $V_1\cup V_2$ forms an independent set, where $V_i=\{v\in V~\vert~f(v)=i\}$, for $i\in \{0,1,2\}$. The total weight of $f$ is equal to $\sum_{v\in V} f(v)$, and is denoted as $w(f)$. The \emph{Independent Roman Domination Number} of $G$, denoted by $i_R(G)$, is defined as min$\{w(f)~\vert~f$ is an IRDF of $G\}$. For a given graph $G$, the problem of computing $i_R(G)$ is defined as the \emph{Minimum Independent Roman Domination problem}. The problem is already known to be NP-hard for bipartite graphs. In this paper, we further study the algorithmic complexity of the problem. 
    In this paper, we propose a polynomial-time algorithm to solve the Minimum Independent Roman Domination problem for distance-hereditary graphs, split graphs, and $P_4$-sparse graphs.

\end{abstract}




\end{frontmatter}

\section{Introduction}
\label{sec:1}
The concept of Roman Dominating function finds its origins in an article authored by Ian Stewart, titled ``Defend the Roman Empire!" \cite{stewart1999defend}, published in Scientific American. In the context of a graph wherein each vertex corresponds to a distinct geographical region within the historical narrative of the Roman Empire, the characterization of a location as secured or unsecured is delineated by the Roman dominating function, denoted as $f$.

Specifically, a vertex $v$ is said to be unsecured if it lacks stationed legions, expressed as $f(v) = 0$. Conversely, a secured location is one where one or two legions are stationed, denoted by $f(v) \in \{1, 2\}$. The strategic methodology for securing an insecure area involves the deployment of a legion from a neighboring location.

In the fourth century A.D., Emperor Constantine the Great enacted an edict precluding the transfer of a legion from a fortified position to an unfortified one if such an action would result in leaving the latter unsecured. Therefore, it is necessary to first have two legions at a given location ($f(v) = 2$) before sending one legion to a neighbouring location. This strategic approach, pioneered by Emperor Constantine the Great, effectively fortified the Roman Empire. Considering the substantial costs associated with legion deployment in specific areas, the Emperor aimed to strategically minimize the number of legions required to safeguard the Roman Empire.

The notion of Roman domination in graphs was first introduced by Cockayne et al. in $2004$. Given a graph $G=(V,E)$, a \emph{Roman dominating function} (RDF) is defined as a function  $f: V \to \{0,1,2\}$, where every vertex $v$, for which $f(v)=0$ must be adjacent to at least one vertex $u$ with $f(u)=2$. The weight of an RDF is defined as $w(f)=\sum_{v\in V} f(v)$. The \emph{Roman Domination Number} is defined as $\gamma_R(G)= min\{w(f)~\vert~f$ is an RDF of $G\}$.

For a graph $G = (V, E)$ and a function $f: V \to \{0,1,2\}$; we define $V_i = \{v\in V~\vert~f(v)=i\}$ for $i\in \{0,1,2\}$. The partition  
 $(V_0,V_1,V_2)$ is said to be  \emph{ordered partition} of $V$ induced by $f$. Note that the function $f:V\to \{0,1,2\}$ and the ordered partition $(V_0,V_1,V_2)$ of $V$ have a one-to-one correspondence. So, when the context is clear, we write $f = (V_0,V_1,V_2)$.

The genesis of the independent dominating set concept can be traced back to chessboard problems. Berge established the formalization of the theory in $1962$. Given a graph $G=(V,E)$, a set $S\subseteq V$ is defined as \emph{independent set} if any two vertices of $S$ are non-adjacent. A set $D\subseteq V$ is said to be a \emph{dominating set} if $N[D]=V$. An \emph{independent dominating set} is a set $S\subseteq V$, such that $S$ is a dominating set as well as an independent set. The independent domination number of $G$, denoted as $i(G)$, is the minimum cardinality of an independent dominating set of $G$.

A function $f$ is referred to as an \emph{Independent Roman dominating function} (\emph{IRDF}) if $f$ is an RDF and $V_1\cup V_2$ is an independent set.  The \emph{Independent Roman Domination Number} is defined as $i_R(G)= min\{w(f)~\vert~f$ is an IRDF of $G\}$. An IRDF $f$ of $G$ with $w(f)=i_R(G)$ is denoted as an $i_R(G)$-function of $G$. Given a graph $G=(V,E)$, the problem of computing $i_R(G)$ is known as \emph{Independent Roman Domination problem}. The problem statement is given below:

\noindent\underline{\textsc{Minimum Independent Roman domination} problem (MIN-IRD)}
\begin{enumerate}
  \item[] \textbf{Instance}: A graph $G=(V,E)$.
  \item[] \textbf{Solution}: $i_R(G)$.
\end{enumerate}

The decision version of the problem is denoted as the DECIDE-IRD problem.

\subsection{Notations and definitions}\label{subsec1.3}
This paper only considers simple, undirected, finite and nontrivial graphs. Let $G=(V,E)$ be a graph. $n$ and $m$ will be used to denote the cardinalities of $V$ and $E$, respectively. $N(v)$ stands for the set of neighbors of a vertex $v$ in $V$. The number of neighbors of a vertex $v\in V$ defines its \emph{degree}, which is represented by the symbol $deg(v)$. The maximum degree of the graph will be denoted by $\Delta$. For a set $U\subseteq V$, the notation $deg_{U}(v)$ is used to represent the number of neighbors that a vertex $v$ has within the subset $U$. Additionally, we use $N_{U}(v)$ to refer to the set of neighbors of vertex $v$ within $U$. Given a set $S\subseteq V$, $G\setminus S$ is defined as the graph induced on $V\setminus S$, that is $G[V\setminus S]$.

A vertex of degree one is known as a \emph{pendant vertex}. A set $I\subseteq V$ is called an \emph{independent set} if no two vertices of $I$ are adjacent. A set $S\subseteq V$ is said to be a \emph{dominating set} if every vertex of $V\setminus S$ is adjacent to some vertex of $S$. A graph 
$G$ is said to be a \emph{complete graph} if any two vertices of $G$ are adjacent. A set $S\subseteq V$ is said to be a clique if the subgraph of $G$ induced on $S$ is a complete graph. For every positive integer $n$,  $[n]$ denotes the set $\{1,2,\ldots,n\}$.

Given a graph $G=(V,E)$ and a function $f:V\to \{0,1,2\}$, $f_H:V(H)\to \{0,1,2\}$ is defined to be the function $f$ restricted on $H$, where $H$ is an induced subgraph of $G$.

The \emph{join} of two graphs $G_{1}$ and $G_{2}$ refers to a graph formed by taking separate copies of $G_{1}$ and $G_{2}$ and connecting every vertex in $V(G_{1})$ to each vertex in $V(G_{2})$ using edges. The symbol $\oplus$ will denote the join operation. Similarly, \emph{disjoint union} of two graphs $H_1$ and $H_2$ is the graph $H=(V(H_1)\cup V(H_2), E(H_1)\cup E(H_2))$. The disjoint union is denoted with the symbol $\cup$.

A vertex $v$ of a graph $G=(V,E)$ is said to be a \emph{universal vertex} if $N[v]=V$. A path with $n$ vertices is denoted as $P_n$. 

A graph $G=(V,E)$ is said to be $P_4$-sparse if a subgraph induced on any $5$ vertices of $G$ contains at most one $P_4$. A \emph{spider} is a graph $G=(V,E)$, where $V$ admits a partition in three subsets $S,C$ and $R$ such that
\begin{itemize}
     \item $C=\{c_1,\ldots,c_l\}~(l\geq 2)$ is a clique.
    \item $S=\{s_1,\ldots,s_l\}$ is an independent set.
    \item Every vertex in $R$ is adjacent to every vertex in $C$ and nonadjacent to all vertex of $S$.
\end{itemize}
A spider $G(S,C,R)$ is said to be a \emph{thin spider} if for every $i\in [l]$, $N(s_i)=\{c_i\}$ and it is called a \emph{thick spider} if for every $i\in [l]$, $N(s_i)=C\setminus\{c_i\}$.

A graph is said to be a \emph{split} graph if it can be partitioned into an independent set and a clique. A \emph{distance-hereditary} graph is a graph in which the distance between any two vertices in any connected induced subgraph is the same as they are in the original graph.

\subsection{Existing Literature}\label{subsec1.1}
The concept of Independent Roman domination was introduced by Cockayne et al. \cite{COCKAYNE200411}. In a private communication, Cockayne and McRae proved that the DECIDE-IRD problem is NP-complete for bipartite graphs (mentioned in  \cite{COCKAYNE200411}). Liu et al. have shown that the MIN-IRD problem is solvable for strongly chordal graphs. In the same article, they have shown that the weighted version of the problem is NP-complete for split graphs \cite{DBLP:journals/jco/LiuC13}, and left an open problem that whether the unweighted MIN-IRD problem is solvable on chordal graphs or not. Padamutham et al. have shown that the problem is NP-complete for dually chordal graphs, comb convex bipartite graphs, and star convex bipartite graphs. In the same article, they have shown that $i_R(G)$ can be computed efficiently for bounded treewidth graphs, chain graphs, and threshold graphs. They have also shown that the MIN-IRD problem is APX hard for graphs with maximum degree $4$ \cite{MR4329924}. Finally, Duan et al. showed that the DECIDE-IRD problem is NP-complete for chordal bipartite graphs. In the same paper, they proposed an efficient algorithm to solve the MIN-IRD problem for trees \cite{DBLP:journals/symmetry/DuanJLWS22}. All the  combinatorial results can be found in \cite{DBLP:journals/ajc/AdabiTRM12, DBLP:journals/dmaa/ChellaliR15,DBLP:journals/dmgt/ChellaliR13,DBLP:journals/access/WuSZJNS18}. In Figure \ref{fig:enter-label}, we show the hierarchy
of some important graph classes and the complexity status of the problem in these graph classes. The graph classes for which we propose an algorithm are highlighted in blue.

\begin{figure}
    \centering
    \includegraphics[scale=0.7]{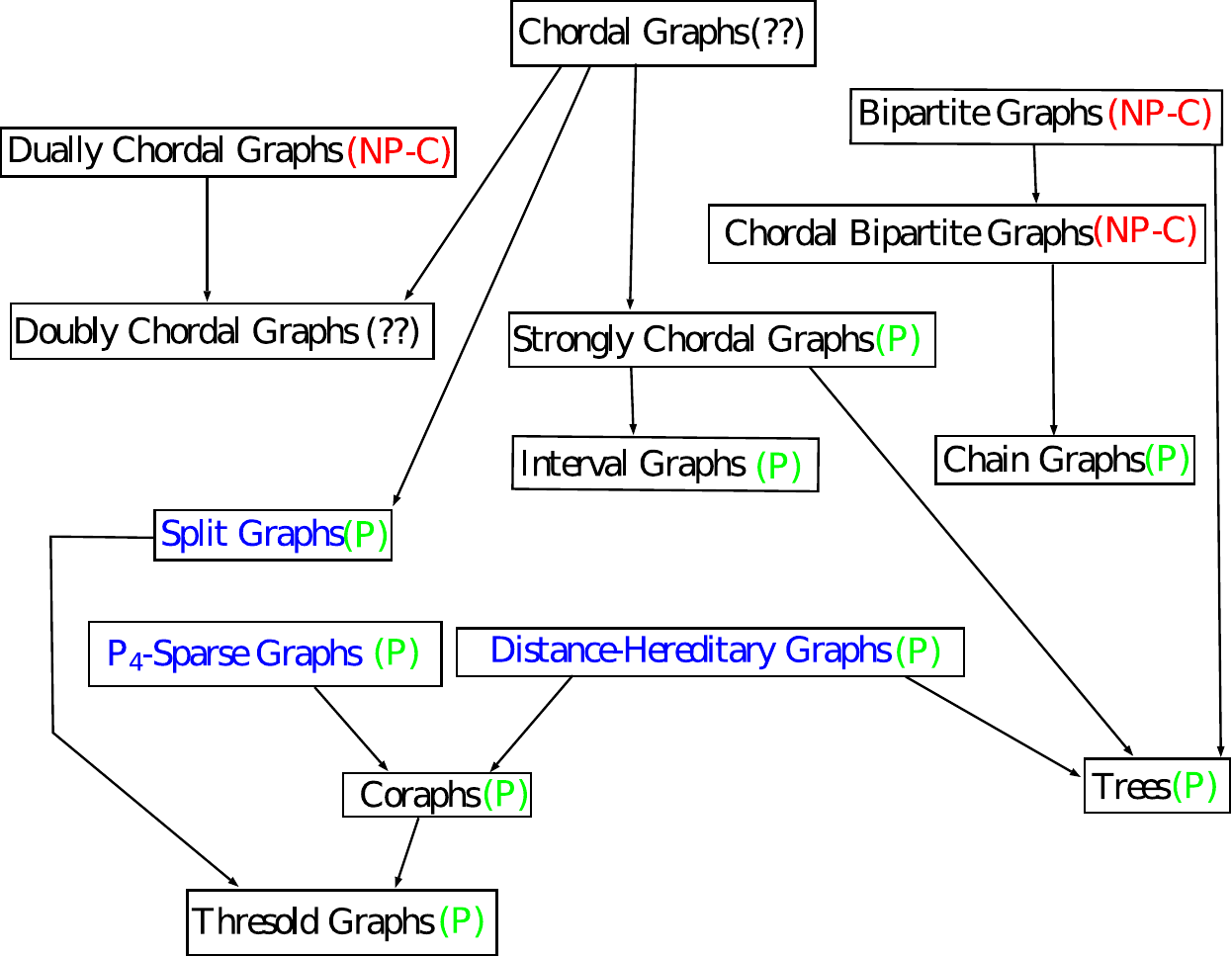}
    \caption{Complexity status of the MIN-IRD problem on some well known graph classes}
    \label{fig:enter-label}
\end{figure}

\subsection{Our Results}\label{subsec1.2}
The subsequent sections of this manuscript are organized in the following manner: In sections \ref{sec:2}, \ref{sec:3}, and \ref{sec:4}, we propose polynomial-time algorithms to solve the MIN-IRD problem for distance-hereditary graphs, split graphs, and $P_4$-sparse graphs respectively. The conclusion of our effort is presented in section \ref{sec:5}.

\section{Algorithm for Distance-Hereditary Graphs}\label{sec:2}
\label{sec:A}
In this section, we assume that $G=(V,E)$ is a distance-hereditary graph. Our objective in this section is to design a linear-time algorithm to compute independent roman domination number for distance-hereditary graphs.

Chang et al. characterized distance-hereditary graphs via edge connections between two special sets of vertices, called twin sets \cite{DBLP:conf/isaac/ChangHC97}. The comprehensive procedure is given in the next paragraph. At its base level, a graph $G$ with a single vertex $v$ is recognized as a distance-hereditary graph, endowed with the twin set $TS(G) = \{v\}$.

A distance-hereditary graph $G$ can be constructed from two existing distance-hereditary graphs, $G_l$ and $G_r$, each possessing twin sets $TS(G_l)$ and $TS(G_r)$, respectively, by using any of the subsequent three operations.

\begin{itemize}
    
    \item In the event that the \emph{true twin} operation is applied to construct the graph $G$ from $G_l$ and $G_r$, then:
    \begin{itemize}
        \item The vertex set of $G$ is  $V(G) = V(G_l) \cup V(G_r)$.
        \item The edge set of $G$ is  $E(G) = E(G_l) \cup E(G_r) \cup \{v_1v_2 \vert v_1 \in TS(G_l), v_2 \in TS(G_r)\}$.
        \item The twin set of $G$ is  $TS(G) = TS(G_l) \cup TS(G_r)$.
    \end{itemize}
    \item In the case where the \emph{false twin} operation is employed to construct the graph $G$ from $G_l$ and $G_r$, then:
    \begin{itemize}
        \item The vertex set of $G$ is  $V(G) = V(G_l) \cup V(G_r)$.
        \item The edge set of $G$ is  $E(G) = E(G_l) \cup E(G_r)$.
        \item The twin set of $G$ is  $TS(G) = TS(G_l) \cup TS(G_r)$.
    \end{itemize}
    \item If the \emph{attachment} operation is employed to construct the graph $G$ from $G_l$ and $G_r$, then:
    \begin{itemize}
        \item The vertex set of $G$ is $V(G) = V(G_l) \cup V(G_r)$.
        \item The edge set of $G$ is $E(G) = E(G_l) \cup E(G_r) \cup \{v_1v_2 ~\vert~v_1 \in TS(G_l), v_2 \in TS(G_r)\}$.
        \item The twin set of $G$ is $TS(G) = TS(G_l)$.
    \end{itemize}
\end{itemize}

\begin{figure}[ht]
\begin{subfigure}{.55\textwidth}
  \centering
  % include first image
  \includegraphics[width=.8\linewidth]{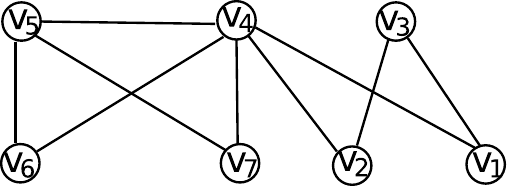}  
  \caption{A distance-hereditary graph $G$}
  \label{fig:sub1}
\end{subfigure}
\begin{subfigure}{.45\textwidth}
  \centering
  % include second image
  \includegraphics[width=1.05\linewidth]{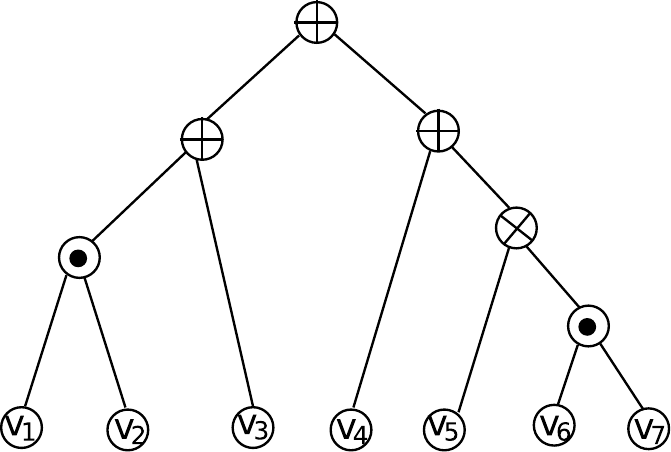}  
  \caption{The decomposition tree $T_G$ of $G$}
  \label{fig:sub2}
\end{subfigure}
\caption{An example of a distance-hereditary graph with its decomposition tree}
\label{fig:3}
\end{figure}

By employing the three operations detailed above, one can systematically construct any distance-hereditary graph. This process leads to the creation of a binary tree representation for a given distance-hereditary graph $G$, commonly referred to as a \emph{decomposition tree}. The definition of this tree is structured as follows: it articulates the sequence of operations through a full binary tree $T$, where the leaves of $T$ correspond to the vertices of $G$. Furthermore, each internal vertex in $T$ is assigned one of the labels $\otimes, \odot,$ or $\oplus$, signifying the true twin operation, false twin operation, and attachment operation, respectively.

In this representation, each leaf of $T$ corresponds to a distance-hereditary graph with a single vertex. A rooted subtree $T'$ of $T$ corresponds to the induced subgraph of $G$ on the vertices represented by the leaves of $T'$. Note that this induced subgraph is itself a distance-hereditary graph. For an internal vertex $v$ of $T$, the label of $v$ corresponds to the operation between the subgraphs represented by the subtrees rooted at the left and right children of $v$. An example is illustrated in Figure \ref{fig:3}.

In the forthcoming proofs, the following definitions are used:
\begin{itemize}
    \item $D(G)=\{g~\vert~g$ is a function from $V$ to $\{0,1,2\}\}$.
    \item $D^2(G)$ = $\{g\in D(G)~\vert~g$ is an IRDF on $G$ such that there exists at least one $u \in TS(G)$ with $g(u)=2 \}$.
    \item $D^1(G)$ = $\{g\in D(G)~\vert~g$ is an IRDF on $G$  such that there exits at least one $u \in TS(G)$ with $g(u)=1$ and $g(v)\neq 2,~\forall~v\in TS(G)\}$.
    \item $D^0(G)$ = $\{g\in D(G)~\vert~g$ is an IRDF on $G$ with $g(TS(G)) = 0\}$.
    \item $D^{00}(G) = \{g \in D(G)~\vert~g$ is an IRDF on $G\setminus TS(G)$ with $g(TS(G)) = 0\}$
\end{itemize}

\noindent In addition, we also define the following  parameters:\\
$u^i(G)=min\{w(g)~\vert~g\in D^i(G)\}$ for $i\in\{0,1,2\}$,  \\$u^{00}(G)=min\{w(g)~\vert~g\in D^{00}(G)\}$. \\

Below, we have listed a couple of observations.

\begin{obs}\label{observation:1}
$i_R(G)$ = min\{$u^0(G), u^1(G), u^2(G)$\}.
\end{obs}

\begin{obs}\label{observation:2}
    If $\vert V\vert=1$, then $u^2(G) = 2, u^1(G) = 1, u^0(G) = \infty, u^{00}(G) = 0. $
\end{obs}

Next, we state few lemmas, which will be helpful in designing the algorithm to compute independent roman domination number  for distance hereditary graphs. 

\begin{lemma}\label{Lemma:1}
Let $G=G_1\otimes G_2$, then the following holds:
\begin{enumerate}
    \item $u^2(G) = min\begin{cases}
    u^2(G_l) + u^{00}(G_r), \\
   u^{00}(G_l) + u^2(G_r).
  \end{cases}$\\
    \item $u^1(G) = min\begin{cases}
    u^1(G_l) + u^0(G_r), \\
   u^0(G_l) + u^1(G_r).
  \end{cases}$

  \item $u^0(G) = u^0(G_l) + u^0(G_r)$.\\
  \item $u^{00}(G) = u^{00}(G_l) + u^{00}(G_r)$.
\end{enumerate}

\end{lemma}

\begin{proof}

\begin{enumerate}
    \item Let $f$ be a member of $D^2(G)$ such that $w(f) = u^2(G)$. By definition, there exists at least one vertex $v\in TS(G)$  with label $2$. Now as $G=G_ l\otimes G_r$, we have $TS(G) = TS(G_l) \cup TS(G_r)$. Now $v$ is contained in either $TS(G_l)$ or $TS(G_r)$. Without loss of generality, let $v\in TS(G_l)$. As $G=G_ l\otimes G_r$, each vertex in $TS(G_r)$ is adjacent to $v$. Hence $f(TS(G_r)) = 0$. Note that $G_l$ is an induced subgraph of $G$. This implies that the function $f$ restricted on $V(G_l)$ must be contained in $D^2(G_l)$. Also, since $f(TS(G_r))=0$, and every vertex of $TS(G_r)$ is adjacent to $v$ which has label $2$, $f$ restricted on $V(G_r)$ is contained in $D^{00}(G_r)$. So, $w(f)=w(f_{G_l})+w(f_{G_r})$, which implies $w(f)=u^2(G)\geq u^2(G_l)+u^{00}(G_r)$. 

Now for the other side of the inequality, let $g_1\in D^2(G_l)$ and $g_2\in D^{00}(G_r)$, such that $w(g_1)=u^2(G_l)$ and $w(g_2)=u^{00}(G_r)$. Now define a function $f:V(G)\to \{0,1,2\}$ as follows:   
 $f(u)=g_1(u)$, for every $u\in V(G_l)$ and $f(v)=g_2(v)$, for every $v\in V(G_r)$. It is easy to observe that $f\in D^2(G)$. Hence $w(f)=w(g_1)+w(g_2)=u^2(G_l)+u^{00}(G_r)$, which implies $u^2(G_l)+u^{00}(G_r)\geq u^2(G)$. Hence $u^2(G)=u^2(G_l)+u^{00}(G_r)$. 

If we consider the case $v\in TS(G_r)$, then $u^2(G)=u^2(G_r)+u^{00}(G_l)$. So, $u^2(G)=min\{u^2(G_l)+u^{00}(G_r), u^2(G_r)+u^{00}(G_l)\}$.

    \item Let $f\in D^1(G)$ such that $w(f) = u^1(G)$. By definition, there exists at least one $v\in TS(G)$ with $f(v) = 1$ and $f(u)\neq 2$,  for all  $u\in TS(G)\setminus \{v\}$. Note that, $TS(G) = TS(G_l) \cup TS(G_r)$. Hence $v$ is contained in either $TS(G_l)$ or $TS(G_r)$. Note that $TS(G_l)$ and $TS(G_r)$ both can not contain vertices with positive labels simultaneously.

Without loss of generality, let $v\in TS(G_l)$. As $G$ is obtained from the true twin operation of $G_l$ and $G_r$, each vertex in $TS(G_r)$ is adjacent to $v$. As $f(v)=1$, $f(TS(G_r)) = 0$. Note that $G_l$ is an induced subgraph of $G$. This implies that the function $f$ restricted on $V(G_l)$ must be contained in $D^1(G_l)$. Again, $G_r$ is an induced subgraph of $G$. Also, there is no $v\in TS(G_l)$ with label $2$, implying that $f$ restricted on $G_r$ is an IRDF on $G_r$. Hence, $f_{G_r}\in D^0(G_r)$. Now $w(f)=w(f_{G_l})+w(f_{G_r})$ implies $w(f)=u^1(G)\geq u^1(G_l)+u^0(G_r)$.

To show the other side of the inequality, let $g_1\in D^1(G_l)$ and $g_2\in D^0(G_r)$, such that $w(g_1)=u^1(G_l)$ and $w(g_2)=u^0(G_r)$. Now we build an IRDF $f$ on $G$ defined as $f(u)=g_1(u),$ for all $u\in G_l$ and  $f(v)=g_2(v),$ for all $v\in G_r$. Now $u^1(G)\leq w(f)=w(g_1)+w(g_2)=u^1(G_l)+u^(G_r)$. So after combining both inequalities, we have $u^1(G)=u^1(G_l)+u^0(G_r)$.

Similarly if $v\in TS(G_r)$, we have $u^1(G)=u^0(G_l)+u^1(G_r)$.  Hence, $u^1(G)=min\{u^1(G_l)+u^0(G_r), u^0(G_l)+u^1(G_r)\}$.

    \item Let $f\in D^0(G)$ with $w(f)=u^0(G).$ As $f(TS(G))=0$, $f(TS(G_l))=f(TS(G_r))=0$. So, labeling of $G_l$ does not have any impact on labeling of $G_r$ (and also vice versa). Hence, we can observe that $f_{G_l}$ and $f_{G_r}$ are IRDF's on $G_l$ and $G_r$ respectively with $f_{G_l}(TS(G_l))=0$ and $f_{G_r}(TS(G_r))=0$. Hence, $f_{G_l}\in D^0(G_l)$ and $f_{G_r}\in D^0(G_r)$. We can also conclude that $w(f_{G_l})=u^0(G_l)$ and $w(f_{G_r})= u^0(G_r)$. If not then for the sake of contradiction, let   $w(f_{G_l})>u^0(G_l)$ or $w(f_{G_r})>u^0(G_r)$. Let $g_1\in D^0(G_l)$ with $w(g_1)= u^0(G_l)$ and $g_2\in D^0(G_r)$ with $w(g_2)= u^0(G_r)$. Now define an IRDF $g$ on $G$ such that $g(v)= g_1(v),$ for all $v\in G_l$ and $g(v)= g_2(v),$ for all $v\in G_r$. Evidently, $g\in D^0(G)$. Note that $w(g)= w(g_1)+w(g_2)=u^0(G_l)+u^0(G_r)< w(f_{G_l})+w(f_{G_r})=w(f)$, which is a contradiction to the fact that $w(f)=u^0(G)$. So, $w(f_{G_l})=u^0(G_l)$ and $w(f_{G_r})= u^0(G_r)$. Hence, $u^0(G) = u^0(G_l) + u^0(G_r)$.
    \item This proof is similar to the proof of part $3$, and hence is omitted.
\end{enumerate}
 
\end{proof}

\begin{lemma}\label{Lemma:2}
    Let $G=G_1\odot G_2$, then the following holds:
    \begin{enumerate}
        \item $u^2(G)= min\begin{cases}
            u^2(G_l) + i_R(G_r), \\
            i_R(G_l) + u^2(G_r).
        \end{cases}$
\\ 
        \item $u^1(G)= min\begin{cases}
            u^1(G_l) + u^1(G_r), \\
            u^1(G_l) + u^0(G_r),  \\
            u^0(G_l) + u^1(G_r).
        \end{cases}$\\
        \item $u^0(G) = u^0(G_l) + u^0(G_r)$.\\
        \item $u^{00}(G) = u^{00}(G_l) + u^{00}(G_r)$.
    \end{enumerate}
\end{lemma}
\begin{proof}
    \begin{enumerate}
        \item Let $f\in D^2(G)$ with $w(f)=u^2(G)$. In the false twin operation, no vertex of $G_l$ is adjacent to any vertex of $G_r$. So the labeling of $G_l$ has no effect on the labeling of $G_r$ and vice-versa. Hence $f_{G_l}$ and $f_{G_r}$ will form two IRDF on $G_l$ and $G_r$ respectively. As $f\in D^2(G)$, there exists at least one vertex in $TS(G)$ with label $2$. Now since, $TS(G)= TS(G_l) \cup TS(G_r)$, two cases can arise. In the first case, $TS(G_l)$ contains at least one vertex with label $2$, and in the second case, $TS(G_r)$ has a vertex with label $2$. 
        
        In the first case, $f_{G_l}\in D^2(G_l)$ and $f_{G_r}$ is an IRDF on $G_r$. Now we show that $w(f_{G_l}) = u^2(G_l)$ and $w(f_{G_r})= i_R(G_r)$. For the sake of contradiction, let  $w(f_{G_l})>u^2(G_l)$ or $w(f_{G_r})>i_R(G_r)$.
        Let $g_1\in D^2(G_l)$ with $w(g_1)= u^2(G_l)$ and an IRDF $g_2\in D(G_r)$ of $G_r$, with $w(g_2)= i_R(G_r)$. Now define an IRDF $g$ on $G$ such that $g(v)= g_1(v),$ for all $v\in G_l$ and $g(v)= g_2(v),$ for all $v\in G_r$. Notice that $w(g)= w(g_1)+w(g_2)=u^2(G_l)+i_R(G_r)< w(f_{G_l})+w(f_{G_r})=w(f)$, which is a contradiction to the fact that $w(f)=u^2(G)$. So, $w(f_{G_l})=u^2(G_l)$ and $w(f_{G_r})= i_R(G_r)$. Hence $u_2(G)=u_2(G_l)+i_R(G_r)$.

        In the other case, it can be similarly shown that $u_2(G)=i_R(G_l)+u_2(G_r)$. Hence, combining both the cases $u_2(G)=min\{u_2(G_l)+i_R(G_r),i_R(G_l)+u_2(G_r)\}$
        
        \item Let $f\in D^2(G)$ with $w(f)=u^1(G)$. Hence, there exists $v\in TS(G)=TS(G_l)\cup TS(G_r)$ such that $f(v)=1$ and $f(x)\neq 2$, for every $x\in TS(G)\setminus \{v\}$. This implies no vertices in $TS(G_l)\cup TS(G_r)$ has label $2$. So, let without loss of generality, $v\in TS(G_l)$, hence $f_{G_l}\in D^1(G_l)$. For the function $f_{G_r}$, all we know is that $f_{G_r}(v)\neq 2$ for every $v\in V(G_r)$. Hence $f_{G_r}\in D^1(G_r)$ or $f_{G_r}\in D^0(G_r)$. Hence using the similar arguments like the proof in part $1$, $u^1(G)=min\{u^1(G_l)+u^1(G_r), u^1(G_l)+u^0(G_r)\}$.

        For the other case, where $v\in TS(G_r)$, $u^1(G)=min\{u^1(G_l)+u^1(G_r), u^0(G_l)+u^1(G_r)\}$. Hence $u^1(G)=min\{u^1(G_l)+u^1(G_r), u^1(G_l)+u^0(G_r),u^0(G_l)+u^1(G_r)\}$.
        
        \item Let $f\in D^0(G)$ with $w(f)=u^0(G)$. Note that $f_{G_l}$ and $f_{G_r}$ both are IRDF on $G_l$ and $G_r$. As $f(TS(G))=0$, we have $f_{G_l}(TS(G_l))=0$ and $f_{G_r}(TS(G_r))=0$ and every vertex in $TS(G_l)$ (or $TS(G_r)$) has a neighbour with label $2$ in $G_l$ (or $G_r$). This implies that $f_{G_l}\in D^0(G_l)$ and $f_{G_r}\in D^0(G_r)$. Again, $w(f_{G_l})$ and $w(f_{G_r})$ are minimum. Otherwise, we can construct an IRDF on $G$ as we constructed in the previous parts whose weight is less than $w(f)$. Also $w(f)= w(f_{G_l})+w(f_{G_r})$, so we can say that $u^0(G) = u^0(G_l) + u^0(G_r)$.\\
        \item The proof of this part is similar to the proof of the previous part.
    \end{enumerate}
\end{proof}

\begin{lemma}\label{Lemma:3}
Let $G=G_1\oplus G_2$, then the following holds: \\
\begin{enumerate}
    \item $u^2(G)= u^2(G_l)+u^{00}(G_r)$
    \item $u^1(G)= u^1(G_l)+u^0(G_r)$
    \item $u^0(G)= min\begin{cases}
            u^{00}(G_l) + u^2(G_r),\\
            u^0(G_l) + u^1(G_r),  \\
            u^0(G_l) + u^0(G_r)  \\
        \end{cases}$\\
    \item $u^{00}(G)= u^{00}(G_l)+i_R(G_r)$   
\end{enumerate}
\end{lemma}
 \begin{proof}

The proofs of parts $1$ and $2$ are similar to the proof of the first two parts of Lemma \ref{Lemma:1}.
      \begin{enumerate}
         \item[3.] Let $f\in D^0(G)$ with $w(f)= u^0(G)$. As $f\in D^0(G)$, $f(TS(G))=f(TS(G_l))=0$. Now, three cases can appear.
         
         In the first case, there exists at least one $v\in TS(G_r)$ such that $f(v)=2$. Hence, by the techniques of the previous proofs, it can be shown that $u^0(G)=u^{00}(G_l)+u^2(G_r)$.

         In the second case, there exists at least one $u\in TS(G_r)$ such that $f(u)=1$ and $f(x)\neq 2$ for every $x\in TS(G_r)$. In this case, by the techniques of the previous proofs, it can be shown that $u^0(G)=u^{0}(G_l)+u^1(G_r)$.

         For the last case, let $f(TS(G_r))=0$. In this case $u^0(G)=u^{0}(G_l)+u^0(G_r)$.

         \item[4.] Let $f\in D^{00}(G)$ with $w(f)= u^{00}(G)$. As $TS(G)=TS(G_l),$ we have $f(TS(G))=f(TS(G_l))=0$. Hence $f_{G_l}\in D^{00}(G)$ and $f_{G_r}\in D(G)$ is an IRDF on $G_r$. By similar techniques used in previous proofs, $u^{00}(G)= u^{00}(G_l)+i_R(G_r)$. 
    \end{enumerate}
 \end{proof}    

 Hence, by using the Lemma 
 \ref{Lemma:1}, \ref{Lemma:2} and \ref{Lemma:3}, Algorithm \ref{ALG:1} can be designed, and the following theorem can be concluded.

\begin{algorithm}[h!]
\footnotesize\textbf{Input:} A distance-hereditary graph $G=(V,E)$, and a decomposition tree $T$ of $G$;\\
\textbf{Output:} $i_R(G)$;\\

Compute a BFS ordering $\sigma = (v_1, v_2,\ldots,v_r)$ of internal vertices of $T$;\\
\For{every leaf $v\in V(T)$}{
    $u^{00}(G)\rightarrow 0$;\\
    $u^0(G)\rightarrow \infty$;\\
    $u^1(G)\rightarrow 1$;\\
    $u^2(G)\rightarrow 2$;\\
    $i_R(G)=1$;\\
}
\For{$i=r$ to $1$}{
    Consider $G_{i}$ as the graph represented by the subtree rooted at $v_i$, featuring two children denoted as $v_l$ and $v_r$;\\
    \uIf{$v_i$ is a vertex with label $\otimes$}{
        $u^2(G_i) = min\{u^2(G_l) + u^{00}(G_r),~u^{00}(G_l) + u^2(G_r)\}$;\\
        $u^1(G_i) = min\{u^1(G_l) + u^0(G_r),~u^0(G_l) + u^1(G_r)\}$;\\
        $u^0(G_i) = u^0(G_l) + u^0(G_r)$;\\
        $u^{00}(G_i) = u^{00}(G_l) + u^{00}(G_r)$;\\
        $i_R(G_i)=min\{u^0(G_i),u^1(G_i),u^2(G_i)\}$;
    }
    \uElseIf{$v_i$ is a node with label $\odot$}{
        $u^2(G_i) = min\{u^2(G_l) + i_R(G_r),~i_R(G_l) + u^2(G_r)\}$;\\
        $u^1(G_i) = min\{u^1(G_l) + u^0(G_r),~u^0(G_l) + u^1(G_r),~u^1(G_l) + u^1(G_r)\}$;\\
        $u^0(G_i) = u^0(G_l) + u^0(G_r)$;\\
        $u^{00}(G_i) = u^{00}(G_l) + u^{00}(G_r)$;\\
        $i_R(G_i)=min\{u^0(G_i),u^1(G_i),u^2(G_i)\}$;
    }
    \Else{
        $u^2(G_i) = u^2(G_l) + u^{00}(G_r)$;\\
        $u^1(G_i) = u^1(G_l) + u^0(G_r)$;\\
        $u^0(G_i) = min\{u^{00}(G_l) + u^2(G_r), u^0(G_l)+u^1(G_r), u^0(G_l)+u^0(G_r)$;\\
        $u^{00}(G_i) = u^{00}(G_l) + i_R(G_r)$;\\
        $i_R(G_i)=min\{u^0(G_i),u^1(G_i),u^2(G_i)\}$;
    }

}
\Return $i_R(G_1)$;

\caption{\textbf{Algorithm to calculate $i_R(G)$ for a distance-hereditary graph $G$}}\label{ALG:1}
\end{algorithm}

\begin{thm}
     Given a distance-hereditary graph $G$ and its decomposition tree $T$, Algorithm \ref{ALG:1} computes $i_R(G)$ in linear time.
 \end{thm}

\newpage

\section{Algorithm for Split Graphs}\label{sec:3}
In the work presented in \cite{COCKAYNE200411}, an open question was raised regarding the solvability of the MIN-IRD problem for chordal graphs. This section partially addresses that question by providing a linear-time algorithm for solving the MIN-IRD problem for a specific subclass of chordal graphs, namely, the class of split graphs.  In \cite{DBLP:journals/jco/LiuC13}, authors have shown that the weighted version of the problem is NP-hard for split graphs. In this section, we show that unweighted version of the problem is efficiently solvable. Throughout this section, we maintain the assumption that $G=(V=K\cup I,E)$ is a connected split graph, where $K$ represents a clique and $I$ denotes an independent set. The subsequent content introduces a couple of noteworthy observations.

\begin{obs}\label{Observation:3}
   If $f:V\to \{0,1,2\}$ is an IRDF of $G$, then $\vert (V_1\cup V_2)\cap K\vert \leq 1$. 
\end{obs}

\begin{obs}\label{Observation:4}
    If $f:V\to \{0,1,2\}$ is an IRDF of $G$, then $V_1\cup V_2$ forms a maximal independent set.
\end{obs}

At first, we compute the independent domination number for a split graph $G=(K\cup I,E)$. Note that an independent dominating set is also a maximal independent set and vice versa. Now, for $G$, the set of all maximal independent sets is $\{I\}\cup\{\{v\}\cup (I\setminus N(v))~\vert~v\in K\}$. So, a minimum maximal independent set is $S_x=\{x\}\cup (I\setminus N(x))$, where $x$ is a maximum degree vertex in $K$. Hence, $S_x$ is also a minimum independent dominating set. Now, we state a theorem from the previous literature.

\begin{thm}\cite{DBLP:journals/ajc/AdabiTRM12}\label{thm_adabi}
    Given a graph $G=(V,E)$, $i_R(G) = i(G)+1$ if and only if $G$ has a vertex of degree $n - i(G)$.
\end{thm}

Now, we state and prove the following important theorem.

\begin{thm}
     Given a split graph $G=(K\cup I,E)$, $i_R(G)=i(G)+1$.
\end{thm}

\begin{proof}
    By Theorem \ref{thm_adabi}, we need to show that $G$ contains a vertex of degree $n-i(G)$. Note that by the previous discussion, $\{x\}\cup (I\setminus N(x))$ is a minimum independent dominating set, where $x$ is a maximum degree vertex of $G$. Hence $\vert \{x\}\cup (I\setminus N(x))\vert=i(G)$, which implies $1+(\vert I\vert-deg_I(x))=i(G)$, so $deg_I(x)=1+\vert I\vert-i(G)$. Hence $deg(x)=deg_I(x)+deg_K(x)=1+\vert I\vert-i(G)+\vert K\vert-1=n-i(G)$. This proves that there exists a vertex with degree $n-i(G)$, hence $i_R(G)=i(G)+1$.
 \end{proof}

 Based on the above discussion, we present the Algorithm \ref{ALG:2} to compute $i_R(G)$ for a split graph $G$.

 \begin{algorithm}[h!]
\footnotesize \textbf{Input:} A split graph $G=(K\cup I,E)$; \\
\textbf{Output:} $i_R(G)$;\\
$Max\leftarrow 0$;\\
$v_{highest\_degree}\leftarrow \bot$;\\
\For{each $x\in K$}{
    \If{$deg(x)>Max$}{
    $Max\leftarrow deg(x)$;\\
    $v_{highest\_degree}\leftarrow x$;\\
    }
}

$i(G)\leftarrow \vert I\vert -Max +1$;\\
$i_R(G)\leftarrow i(G)+1$;\\
\Return $i_R(G)$;\\
\caption{\textbf{Algorithm to calculate $i_R(G)$ for a split graph $G$ ($ALG\_SPLIT(G)$)}}\label{ALG:2}
\end{algorithm}

Depending on the above algorithm, we define $f:V \to \{0,1,2\}$ in following way:
$f(v_{highest\_degree})=2$, $f(x)=1$ for $x\in N(v_{highest\_degree})$ and $f(x)=0$, otherwise. Note that $f$ is an IRDF with $w(f)=i_R(G)$, hence $f$ is an $i_R(G)$-function.

\begin{thm}\label{Theorem:3}
    Given a split graph $G$, Algorithm \ref{ALG:2} computes $i_R(G)$ in linear time.
\end{thm}

\section{Algorithm for $P_4$-sparse graphs}
\label{sec:4}
In this section, we present an algorithm that efficiently solves the MIN-IRD problem for $P_4$-sparse graphs. The class of $P_4$-sparse graphs is an extension of the class of cographs. Below, we recall a characterization theorem for $P_4$-sparse graphs.

\begin{thm}\cite{Defn_P4sparse}\label{th:5}
    A graph $G$ is said to be $P_4$-sparse if and only if one of the following conditions hold
    \begin{itemize}
        \item $G$ is a single vertex graph.
        \item $G=G_1\cup G_2$, where $G_1$ and $G_2$ are $P_4$-sparse graphs.
        \item $G=G_1 \oplus G_2$, where $G_1$ and $G_2$ are $P_4$-sparse graphs.
        \item $G$ is a spider (thick or thin) that admits a spider partition $(S, C, R)$ where either $G[R]$ is a $P_4$-sparse graph or $R=\phi$.
    \end{itemize} 
\end{thm}

Hence, by Theorem \ref{th:5}, a graph that is $P_4$-sparse and contains at least two vertices can be classified as either a join or union of two $P_4$-sparse graphs or a particular type of spider (thick or thin). Consequently, in this section, we compute the Independent Roman domination number for each of these scenarios. Note that if $G$ is a thick (or thin) spider without a head, it becomes a split graph. As proved in the previous section, the problem can be solved in linear time for split graphs. Therefore, our focus here is on cases where $G$ is a spider with a non empty head.

\begin{lemma}\label{lem:4}
    Let $G=(V,E)$ be a thin spider with a nonempty head, ($R\neq\phi$ in the spider partition $(S,C,R)$). Then, $i_R(G)$ is equal to $\vert C\vert+1$.
\end{lemma}

\begin{proof}
     Note that for the graph $G$, $i(G)=\vert C\vert$, where $\{c_i\}\cup (S\setminus\{s_i\})$ forms a minimum independent dominating set. To complete the proof of this lemma, we establish the following two statements:
    \begin{itemize}
        \item[i.] There exists an IRDF $f$ on $G$ with $w(f)=\vert C\vert +1$.
        \item[ii.] $i(G)<i_R(G)$.
    \end{itemize}
    
    Firstly, we define an IRDF with a weight of $\vert C\vert +1$. Consider the function $f:V\rightarrow \{0,1,2\}$ as follows: $f(c_i)=2$ for a fixed $i\in [k]$, and $f(s_j)=1$ for all $s_j$ in the set $S\setminus {s_i}$. For any other element $u\in V$, $f(u)=0$. It can be observed that the function $f$ satisfies the properties of being an IRDF with a weight of $\vert C\vert+1$.

    Now we show that $i(G)<i_R(G)$. For the sake of contradiction, assume that $i(G)=i_R(G)=\vert C\vert$. Then, there exists an IRDF $f$ of $G$ with $w(f)=i(G)=\vert C\vert$. Since $V_1\cup V_2$ forms an independent dominating set, $\vert V_1\cup V_2\vert \geq i(G)$. But $\vert V_1\vert +2\vert V_2\vert=i(G)$, implying $V_2=\emptyset$. But there exist vertices with label zero since $\vert V_0\vert=\vert V\vert-\vert V_1\cup V_2\vert=\vert V\vert-\vert C\vert>0$. This leads to a contradiction since no vertex with label $0$ has a neighbour with label $2$. Hence $i(G)<i_R(G)$. This implies that $i_R(G)=\vert C\vert+1$.
\end{proof}

 \begin{lemma}\label{lem:5}
    Let $G=(V,E)$ be a thick spider with a nonempty head, ($R\neq\phi$ in the spider partition $(S,C,R)$). Then, $\gamma_{r} (G)$ is equal to $3$.
\end{lemma}

\begin{proof}

Note that $i(G)=2$, where $\{c_1,s_1\}$ forms a minimum independent dominating set. Now, to complete the proof of this lemma, we show the following two statements:

\begin{itemize}

\item[i.] There exists an IRDF $f$ on $G$ with $w(f)=3$.
\item[ii.] $i(G)<i_R(G)$.

\end{itemize}

Initially, we establish an IRDF with a weight of $3$. Consider the function $f:V\rightarrow \{0,1,2\}$, defined as follows: $f(c_i)=2$ for a fixed $i\in [k]$ and $f(s_i)=1$. For any other element $u$, $f(u)=0$. It can be observed that the function $f$ is an IRDF of $G$ with $w(f)=3$.

Now we show that $i(G)<i_R(G)$. For the sake of contradiction, assume that $i(G)=i_R(G)=2$. Then, there exists an IRDF $f$ of $G$ with $w(f)=i(G)=2$. Since $V_1\cup V_2$ forms an independent dominating set, $\vert V_1\cup V_2\vert =\vert V_1\vert +\vert V_2\vert\geq i(G)=2$. But $\vert V_1\vert +2\vert V_2\vert=i(G)$, which implies $V_2=\emptyset$. But there exist vertices with label zero since $\vert V_0\vert\geq \vert V\vert-\vert C\vert>0$. This leads to a contradiction since no vertex with label $0$ has a neighbour with label $2$. Hence $i(G)<i_R(G)$. This implies that $i_R(G)=3$.
\end{proof}

Now, when $G$ is the disjoint union of two $P_4$-sparse graphs $G_1$ and $G_2$, then the next observation follows.

\begin{obs}\label{Observation:5}
Let $G$ be the disjoint union of two $P_4$-sparse graphs $G_1$ and $G_2$. Then $i_R(G)=i_R(G_1)+i_R(G_2)$.
\end{obs}

Now we are left with the case when $G$ is the join of two $P_4$-sparse graphs $G_1$ and $G_2$. Considering this case, we propose an easy-to-follow observation.

\begin{obs}\label{Observation:6}
Let $G$ be the join of two $P_4$-sparse graphs $G_1$ and $G_2$ and $f$ be an IRDF of $G$. Then either $f(V(G_1))=0$ and $f_{G_2}$ is an IRDF of $G_2$ or $f(V(G_2))=0$ and $f_{G_1}$ is an IRDF of $G_1$.
\end{obs}

Now, with the help of Observation \ref{Observation:6}, we prove the following lemma.

\begin{lemma}
    Let $G$ be the join of two $P_4$-sparse graphs $G_1$ and $G_2$, then
    \begin{enumerate}
        \item $i_R(G)=min\{\vert V(G_1)\vert, \vert V(G_2)\vert\}+1$, if $E(G_1)=E(G_2)=\emptyset$.
        \item $i_R(G)=min\{\vert V(G_1)\vert+1,i_R(G_2)\}$, if $E(G_1)=\emptyset$ and $E(G_2)\neq\emptyset$.
        \item $i_R(G)=min\{i_R(G_1),\vert V(G_2)\vert+1\}$, if $E(G_1)\neq\emptyset$ and $E(G_2)=\emptyset$.
        \item $i_R(G)=min\{i_R(G_1),i_R(G_2)\}$, if $E(G_1)\neq \emptyset$ and $E(G_2)\neq\emptyset$
    \end{enumerate}
\end{lemma}

\begin{proof}
    \begin{enumerate}
        \item Let $E(G_1)=E(G_2)=\emptyset$. This implies that $V(G_1)$ and $V(G_2)$ are independent sets. Hence, $G$ is a complete bipartite graph. Let $f$ be an $i_R(G)$-function of $G$. By Observation \ref{Observation:6}, either $f(V(G_1))=0$ and $f_{G_2}$ is an IRDF of $G_2$ or $f(V(G_2))=0$ and $f_{G_1}$ is an IRDF of $G_1$. It is easy to observe that if  $f(V(G_1))=0$ and $f_{G_2}$ is an IRDF of $G_2$ then $w(f)=\vert V(G_2)\vert+1$. And for the latter $w(f)=\vert V(G_1)\vert+1$. This implies that $i_R(G)=min\{\vert V(G_1)\vert, \vert V(G_2)\vert\}+1$.
        
        \item Now let $E(G_1)=\emptyset$, $E(G_2)\neq\emptyset$ and $f$ be an $i_R(G)$-function of $G$. By Observation \ref{Observation:6}, either $f(V(G_1))=0$ and $f_{G_2}$ is an IRDF of $G_2$ or $f(V(G_2))=0$ and $f_{G_1}$ is an IRDF of $G_1$. 
        
        Let $f(V(G_1))=0$ and $f_{G_2}$ is an IRDF of $G_2$. This implies $i_R(G)\geq i_R(G_2)$. Now consider any $i_R(G_2)$-function $f'$ on $G_2$. Since $E(G_2)\neq \emptyset$, then there exists $v\in V(G_2)$ such that $f'(v)=2$. Now define $g:V(G)\to \{0,1,2\}$ as follows: $g(u)=f'(u)$, for every $u\in V(G_2)$ and $g(u)=0$ for every $u\in V(G_1)$. Evidently, $g$ is an IRDF of $G$. This implies that $i_R(G)\leq w(g)=w(f')=i_R(G_2)$. Hence, $i_R(G)=i_R(G_2)$. 

        Let $f(V(G_2))=0$ and $f_{G_1}$ is an IRDF of $G_1$. Note that $f(v)\geq 1$ for every $v\in V(G_1)$, as $f(N_G(v))=f(V(G_2))=0$.  But there must exist a vertex $v_1\in V(G_1)$ such that $f(v_1)=2$ or the vertices in $V(G_2)$ would not have any neighbour with label $2$. This implies that $i_R(G)\geq \vert V(G_1)\vert$+1. Now define $g:V(G)\to \{0,1,2\}$ as follows: $g(v)=2$, for some fixed $v\in V(G_1)$, $g(u)=1$ for every $u\in V(G_1)\setminus \{v\}$ and $g(u)=0$ for every $u\in V(G_2)$. Evidently, $g$ is an IRDF of $G$. This implies that $i_R(G)\leq w(g)=\vert V(G_1)\vert +1$. Hence, $i_R(G)=\vert V(G_1)\vert +1$. So, $i_R(G)=min\{\vert V(G_1)\vert+1,i_R(G_2)\}$.

        \item The proof of part $3$ is similar to part $2$.

        \item The proof of part $4$ is similar to part $2$.
    \end{enumerate} 
\vspace*{-.5cm}
\end{proof}

Now, we are ready to present Algorithm \ref{ALG:3}.
%and 
%Theorem \ref{Theorem:5} can be concluded.

\begin{algorithm}[h!]
 \footnotesize \textbf{Input:} A $P_4$-sparse graph $G=(V,E)$; \\
\textbf{Output:} $i_R(G)$;\\

\uIf{$(G$ is a thin spider with spider partition $(S,C,R))$}{
    \uIf{$R=\emptyset$}{
        \Return $ALG\_SPLIT(G)$;
    }\Else{
        \Return $\vert C\vert+1$;
    }
}
\uElseIf{$(G$ is a thick spider with spider partition $(S,C,R))$}{
        \uIf{$R=\emptyset$}{
        \Return $ALG\_SPLIT(G)$;
    }\Else{
        \Return $3$;
    }
}
\uElseIf{$(G$ is disjoint union of two $P_4$-sparse graphs $G_1$ and $G_2)$}{
    \Return $ALG\_P4\_SPARSE(G_1)+ALG\_P4\_SPARSE(G_2)$;    
}
\Else{
    Let $G$ is the join of two $P_4$-sparse graphs $G_1$ and $G_2$;\\
    \uIf{$(\vert E(G_1)\vert=0$ and $\vert E(G_2)\vert=0)$}{
        \Return $min\{\vert V(G_1)\vert, \vert V(G_2)\vert\}+1$; 
    }\uElseIf{$(\vert E(G_1)\vert=0$ and $\vert E(G_2)\vert\neq 0)$}{
        \Return $min\{\vert V(G_1)\vert+1,ALG\_P4\_SPARSE(G_2)\}$;
    }\uElseIf{$(\vert E(G_1)\vert\neq 0$ and $\vert E(G_2)\vert= 0)$}{
        \Return $min\{\vert V(G_2)\vert+1,ALG\_P4\_SPARSE(G_1)\}$;
    }\Else{
        \Return $min\{ALG\_P4\_SPARSE(G_1),ALG\_P4\_SPARSE(G_2)\}$;
    }
}

\caption{\textbf{Algorithm to calculate $i_R(G)$ for a $P_4$-sparse graph $G$ ($ALG\_P4\_SPARSE(G)$)}}\label{ALG:3}
\end{algorithm}

\newpage
\begin{thm}\label{Theorem:5}
    Given a $P_4$-sparse graph $G$, $ALG\_P4\_SPARSE(G)$ computes $i_R(G)$ in polynomial-time.
\end{thm}

\section{Conclusion}
\label{sec:5}
In this work, we have extended the algorithmic study of the Independent Roman Domination problem. We have proposed efficient algorithms for the problem on three important graph classes. The question still remains whether there exists any efficient algorithm that solves the problem for chordal graphs or not. The problem can also be studied in other interesting graph classes, for example  AT-free graphs and doubly chordal graphs.

 \bibliographystyle{elsarticle-num} 
 \bibliography{cas-refs}

\begin{thebibliography}{10}
\expandafter\ifx\csname url\endcsname\relax
  \def\url#1{\texttt{#1}}\fi
\expandafter\ifx\csname urlprefix\endcsname\relax\def\urlprefix{URL }\fi
\expandafter\ifx\csname href\endcsname\relax
  \def\href#1#2{#2} \def\path#1{#1}\fi

\bibitem{stewart1999defend}
I.~Stewart, Defend the roman empire!, Scientific American 281 (1999) 136--138.

\bibitem{COCKAYNE200411}
E.~J. Cockayne, P.~A. Dreyer, S.~M. Hedetniemi, S.~T. Hedetniemi, Roman domination in graphs, Discrete Mathematics 278 (2004) 11--22.

\bibitem{DBLP:journals/jco/LiuC13}
C.~Liu, G.~J. Chang, Roman domination on strongly chordal graphs, J. Comb. Optim. 26~(3) (2013) 608--619.

\bibitem{MR4329924}
C.~Padamutham, V.~S. Reddy~Palagiri, Complexity aspects of variants of independent {R}oman domination in graphs, Bull. Iranian Math. Soc. 47~(6) (2021) 1715--1735.

\bibitem{DBLP:journals/symmetry/DuanJLWS22}
Z.~Duan, H.~Jiang, X.~Liu, P.~Wu, Z.~Shao, Independent roman domination: The complexity and linear-time algorithm for trees, Symmetry 14~(2) (2022) 404.

\bibitem{DBLP:journals/ajc/AdabiTRM12}
M.~Adabi, E.~E. Targhi, N.~J. Rad, M.~S. Moradi, Properties of independent roman domination in graphs, Australas. {J} Comb. 52 (2012) 11--18.

\bibitem{DBLP:journals/dmaa/ChellaliR15}
M.~Chellali, N.~J. Rad, Trees with independent roman domination number twice the independent domination number, Discret. Math. Algorithms Appl. 7~(4) (2015) 1550048:1--1550048:8.

\bibitem{DBLP:journals/dmgt/ChellaliR13}
M.~Chellali, N.~J. Rad, Strong equality between the roman domination and independent roman domination numbers in trees, Discuss. Math. Graph Theory 33~(2) (2013) 337--346.

\bibitem{DBLP:journals/access/WuSZJNS18}
P.~Wu, Z.~Shao, E.~Zhu, H.~Jiang, S.~Nazari{-}Moghaddam, S.~M. Sheikholeslami, Independent roman domination stable and vertex-critical graphs, {IEEE} Access 6 (2018) 74737--74746.

\bibitem{DBLP:conf/isaac/ChangHC97}
M.~Chang, S.~Hsieh, G.~Chen, Dynamic programming on distance-hereditary graphs, in: H.~W. Leong, H.~Imai, S.~Jain (Eds.), Algorithms and Computation, 8th International Symposium, {ISAAC} '97, Singapore, December 17-19, 1997, Proceedings, Vol. 1350 of Lecture Notes in Computer Science, Springer, 1997, pp. 344--353.

\bibitem{Defn_P4sparse}
G.~Bagan, H.~B. Merouane, M.~Haddad, H.~Kheddouci, On some domination colorings of graphs, Discret. Appl. Math. 230 (2017) 34--50.

\end{thebibliography}





% \end{thebibliography}
\end{document}